\documentclass[12pt,reqno]{amsart}
\usepackage{amsfonts}
\usepackage{amsmath}
\usepackage{amssymb,latexsym}
\usepackage{enumerate}
\usepackage[bookmarksnumbered, colorlinks, plainpages]{hyperref}
\usepackage{amsbsy}
\usepackage{amscd}
\usepackage{epsfig}
\usepackage{graphicx}
\usepackage{epstopdf}
\usepackage{caption}
\usepackage{subcaption}

\newtheorem{theorem}{Theorem}[section]
\newtheorem{lemma}{Lemma}[section]

\newtheorem{definition}{Definition}[section]

\newtheorem{example}{Example}[section]
\newtheorem{remark}{Remark}[section]
\numberwithin{equation}{section}

\begin{document}
\date{{\scriptsize Received: , Accepted: .}}
\title[Fixed-Circle Problem]{Fixed-Circle Problem on $S$-Metric Spaces with
a Geometric Viewpoint}
\author[N. Y. \"{O}ZG\"{U}R ]{N\.{I}HAL YILMAZ \"{O}ZG\"{U}R}
\address{Bal\i kesir University\\
Department of Mathematics\\
10145 Bal\i kesir, TURKEY}
\email{nihal@balikesir.edu.tr}
\author[N. TA\c{S}]{N\.{I}HAL TA\c{S}}
\address{Bal\i kesir University\\
Department of Mathematics\\
10145 Bal\i kesir, TURKEY}
\email{nihaltas@balikesir.edu.tr}
\maketitle

\begin{abstract}
Recently, a new geometric approach which is called the fixed-circle problem
has been gained to fixed-point theory. The problem is introduced and studied
using different techniques on metric spaces. In this paper, we consider the
fixed-circle problem on $S$-metric spaces. We investigate existence and
uniqueness conditions for fixed circles of self-mappings on an $S$-metric
space. Some examples of self-mappings having fixed circles are also given.%
\end{abstract}




%

%

\section{\textbf{Introduction}}

\label{intro}

It has been extensively studied the existence and uniqueness theorems of
fixed points satisfy some contractive conditions since the time of Stefan
Banach (see \cite{Banach} and \cite{Ciesielski-2007}). Then many authors
have been investigated new fixed point theorems on metric spaces or
generalizations of metric spaces such as $S$-metric spaces. For example,
Sedghi, Shobe and Aliouche obtained the Banach's contraction principle on $S$%
-metric spaces in \cite{Sedghi-2012}. The present authors studied some
generalizations of the Banach's contraction principle on $S$-metric spaces
in \cite{nihal3}. Also they investigated new fixed point theorems for the
following contractive condition (which is called Rhoades' condition \cite%
{Rhoades}) on $S$-metric spaces in \cite{nihal} and \cite{tez}:%
\begin{eqnarray*}
(S25)\text{ \ \ }\mathcal{S}(Tx,Tx,Ty) &<&\max \{\mathcal{S}(x,x,y),\mathcal{%
S}(Tx,Tx,x),\mathcal{S}(Ty,Ty,y), \\
&&\mathcal{S}(Ty,Ty,x),\mathcal{S}(Tx,Tx,y)\}\text{,}
\end{eqnarray*}%
for each $x,y\in X$, $x\neq y$. Then they gave the concept of diameter on $S$%
-metric spaces and obtained a new contractive condition using this notion as
follows \cite{nihal}:%
\begin{equation*}
(S25a)\text{ \ \ }\mathcal{S}(Tx,Tx,Ty)<diam\{U_{x}\cup U_{y}\}\text{,}
\end{equation*}%
for each $x,y\in X$ $(x\neq y)$, where $U_{x}=\{T^{n}x:n\in
\mathbb{N}
\}$, $U_{y}=\{T^{n}y:n\in
\mathbb{N}
\}$, $diam\{U_{x}\}<\infty $ and $diam\{U_{y}\}<\infty $.

Although it has been studied the existence of fixed points of functions on
various metric spaces but there is no study on the existence of fixed
circles. The fixed circle problem arises naturally. There are some examples
of functions with a fixed circle on some special metric spaces. For example,
let $%
\mathbb{C}
$ be an $S$-metric space with the $S$-metric%
\begin{equation*}
\mathcal{S}(z,w,t)=\frac{\left\vert z-t\right\vert +\left\vert
w-t\right\vert }{2}\text{,}
\end{equation*}%
for all $z,w,t\in
\mathbb{C}
$. Let the mapping $T$ be defined as%
\begin{equation*}
Tz=\frac{1}{\overline{z}}\text{,}
\end{equation*}%
for all $z\in
\mathbb{C}
\setminus \{0\}$. The mapping $T$ fixes the unit circle $C_{0,1}^{S}=\left\{
x\in X:\mathcal{S}(x,x,0)=1\right\} $.

Recently, \"{O}zdemir, \.{I}skender and \"{O}zg\"{u}r used new types of
activation functions having a fixed circle for a complex valued neural
network \cite{Ozdemir-2011}. The usage of these types activation functions
lead us to guarantee the existence of fixed points of the complex valued
Hopfield neural network (see \cite{Ozdemir-2011} for more details).

Hence it is important to investigate some fixed-circle theorems on various
metric spaces. In \cite{nihal4}, the present authors obtained some
fixed-circle theorems on metric spaces. They studied some existence theorems
for fixed circles with a geometric interpretation and gave necessary
conditions for the uniqueness of fixed circles on metric spaces. Also they
gave some examples of self-mappings with fixed circles. On the other hand,
in \cite{Ozgur-Aip}, they proved new fixed-circle results and gave to an
application of the obtained results to discontinuity problem and
discontinuous activation functions.

Motivated by the above studies, our aim in this paper is to obtain some
fixed-circle theorems for self-mappings on $S$-metric spaces. In Section \ref%
{sec:1} we recall some definitions, lemmas and basic facts about $S$-metric
spaces. In Section \ref{sec:2} we introduce the notion of a fixed circle on $%
S$-metric spaces. Therefore we obtain some existence and uniqueness theorems
for self-mappings having fixed circles via different techniques. We
investigate the case in which the number of the fixed circles are infinitely
many. Some examples of self-mappings with fixed circles are given with a
geometric viewpoint. Using Mathematica (Wolfram Research, Inc., Mathematica,
Trial Version, Champaign, IL (2016)) we draw some figures.

\section{\textbf{Preliminaries}}

\label{sec:1} In this section we recall some definitions and basic facts
about $S$-metric spaces. At first, we recall the concept of an $S$-metric
space.

\begin{definition}
\label{def1} \cite{Sedghi-2012} Let $X$ be nonempty set and $\mathcal{S}%
:X^{3}\rightarrow \lbrack 0,\infty )$ be a function satisfying the following
conditions for all $x,y,z,a\in X$.

\begin{enumerate}
\item $\mathcal{S}(x,y,z)=0$ if and only if $x=y=z$,

\item $\mathcal{S}(x,y,z)\leq \mathcal{S}(x,x,a)+\mathcal{S}(y,y,a)+\mathcal{%
S}(z,z,a)$.
\end{enumerate}

Then $S$ is called an $S$-metric on $X$ and the pair $(X,\mathcal{S})$ is
called an $S$-metric space.
\end{definition}

The following lemma can be considered the symmetry condition and it is used
in the proof of some theorems.

\begin{lemma}
\label{lem1} \cite{Sedghi-2012} Let $(X,\mathcal{S})$ be an $S$-metric
space. Then%
\begin{equation*}
\mathcal{S}(x,x,y)=\mathcal{S}(y,y,x)\text{.}
\end{equation*}
\end{lemma}

The relationships between a metric and an $S$-metric was given in the
following lemma.

\begin{lemma}
\label{lem2} \cite{Hieu} Let $(X,d)$ be a metric space. Then the following
properties are satisfied$:$

\begin{enumerate}
\item $\mathcal{S}_{d}(x,y,z)=d(x,z)+d(y,z)$ for all $x,y,z\in X$ is an $S$%
-metric on $X$.

\item $x_{n}\rightarrow x$ in $(X,d)$ if and only if $x_{n}\rightarrow x$ in
$(X,\mathcal{S}_{d})$.

\item $\{x_{n}\}$ is Cauchy in $(X,d)$ if and only if $\{x_{n}\}$ is Cauchy
in $(X,\mathcal{S}_{d}).$

\item $(X,d)$ is complete if and only if $(X,\mathcal{S}_{d})$ is complete.
\end{enumerate}
\end{lemma}

The metric $\mathcal{S}_{d}$ was called as the $S$-metric generated by $d$
\cite{nihal2}. We know some examples of an $S$-metric which is not generated
by any metric (see \cite{Hieu}, \cite{nihal2} and \cite{tez} for more
details).

The notions of an open ball, a closed ball and diameter were introduced on $%
S $-metric spaces as the following definitions.

\begin{definition}
\label{def2} \cite{Sedghi-2012} Let $(X,\mathcal{S})$ be an $S$-metric
space. The open ball $B_{S}(x_{0},r)$ and closed ball $B_{S}[x_{0},r]$ with
a center $x_{0}$ and a radius $r$ are defined by%
\begin{equation*}
B_{S}(x_{0},r)=\{x\in X:\mathcal{S}(x,x,x_{0})<r\}
\end{equation*}%
and%
\begin{equation*}
B_{S}[x_{0},r]=\{x\in X:\mathcal{S}(x,x,x_{0})\leq r\}\text{,}
\end{equation*}%
for $r>0$ and $x_{0}\in X$.
\end{definition}

\begin{definition}
\label{def5} \cite{nihal} Let $(X,\mathcal{S})$ be an $S$-metric space and $%
A $ be a nonempty subset of $X$. The diameter of $A$ is defined by%
\begin{equation*}
diam\{A\}=sup\{\mathcal{S}(x,x,y):x,y\in A\}\text{.}
\end{equation*}%
If $A$ is $S$-bounded, then we will write $diam\{A\}<\infty $.
\end{definition}

Now we define the notion of a circle on an $S$-metric space.

\begin{definition}
\label{def3} Let $(X,\mathcal{S})$ be an $S$-metric space and $x_{0}\in X$, $%
r\in (0,\infty )$. We define the circle centered at $x_{0}$ with radius $r$
as%
\begin{equation*}
C_{x_{0},r}^{S}=\{x\in X:\mathcal{S}(x,x,x_{0})=r\}\text{.}
\end{equation*}
\end{definition}

\section{\textbf{Some Fixed-Circle Theorems on $S$-Metric Spaces}}

\label{sec:2} In this section we introduce the notion of a fixed circle on
an $S$-metric space. Then we investigate some existence and uniqueness
theorems for self-mappings having fixed circles.

\begin{definition}
\label{def4} Let $(X,\mathcal{S})$ be an $S$-metric space, $%
C_{x_{0},r}^{S}=\{x\in X:\mathcal{S}(x,x,x_{0})=r\}$ be a circle on $X$ and $%
T:X\rightarrow X$ be a self-mapping. If $Tx=x$ for all $x\in C_{x_{0},r}^{S}$
then the circle $C_{x_{0},r}^{S}$ is said to be a fixed circle of $T$.
\end{definition}

\subsection{\textbf{The existence of fixed circles}}

In this section we obtain some existence theorems for fixed circles of
self-mappings.

\begin{theorem}
\label{thm1} Let $(X,\mathcal{S})$ be an $S$-metric space and $%
C_{x_{0},r}^{S}$ be any circle on $X$. Let us define the mapping
\begin{equation}
\varphi :X\rightarrow \lbrack 0,\infty )\text{, }\varphi (x)=\mathcal{S}%
(x,x,x_{0})\text{,}  \label{phi mapping}
\end{equation}%
for all $x\in X$. If there exists a self-mapping $T:X\rightarrow X$
satisfying%
\begin{equation}
\mathcal{S}(x,x,Tx)\leq \varphi (x)+\varphi (Tx)-2r  \label{thm1_S1}
\end{equation}%
and%
\begin{equation}
\mathcal{S}(x,x,Tx)+\mathcal{S}(Tx,Tx,x_{0})\leq r\text{,}  \label{thm1_S2}
\end{equation}%
for all $x\in C_{x_{0},r}^{S}$, then $C_{x_{0},r}^{S}$ is a fixed circle of $%
T$.
\end{theorem}

\begin{proof}
Let $x\in C_{x_{0},r}^{S}$. Then using the conditions (\ref{thm1_S1}), (\ref%
{thm1_S2}), Lemma \ref{lem1} and the triangle inequality, we get%
\begin{eqnarray*}
\mathcal{S}(x,x,Tx) &\leq &\varphi (x)+\varphi (Tx)-2r \\
&=&\mathcal{S}(x,x,x_{0})+\mathcal{S}(Tx,Tx,x_{0})-2r \\
&\leq &\mathcal{S}(x,x,Tx)+\mathcal{S}(x,x,Tx)+\mathcal{S}(Tx,Tx,x_{0})+%
\mathcal{S}(Tx,Tx,x_{0})-2r \\
&=&2\mathcal{S}(x,x,Tx)+2\mathcal{S}(Tx,Tx,x_{0})-2r \\
&\leq &2r-2r=0
\end{eqnarray*}%
and so%
\begin{equation*}
\mathcal{S}(x,x,Tx)=0\text{,}
\end{equation*}%
which implies $Tx=x$. Consequently, $C_{x_{0},r}^{S}$ is a fixed circle of $%
T $.
\end{proof}

\begin{remark}
\label{rem1} $1)$ Notice that the condition $($\ref{thm1_S1}$)$ guarantees
that $Tx$ is not in the interior of the circle $C_{x_{0},r}^{S}$ for $x\in
C_{x_{0},r}^{S}$. Similarly the condition $($\ref{thm1_S2}$)$ guarantees
that $Tx$ is not exterior of the circle $C_{x_{0},r}^{S}$ for $x\in
C_{x_{0},r}^{S}$. Hence $Tx\in C_{x_{0},r}^{S}$ for each $x\in
C_{x_{0},r}^{S}$ and so we get $T(C_{x_{0},r}^{S})\subset C_{x_{0},r}^{S}$.

$2)$ If an $S$-metric is generated by any metric $d$, then Theorem \ref{thm1}
can be used on the corresponding metric space.
\end{remark}

Now we give an example of a self-mapping with a fixed circle.

\begin{example}
\label{exm6} Let $X=%
\mathbb{R}
$ and the function $\mathcal{S}:X^{3}\rightarrow \lbrack 0,\infty )$ be
defined by%
\begin{equation*}
\mathcal{S}(x,y,z)=\left\vert x-z\right\vert +\left\vert y-z\right\vert
\text{,}
\end{equation*}%
for all $x,y,z\in
\mathbb{R}
$ \cite{Sedghi-2014}. Then $(X,\mathcal{S})$ is called the usual $S$-metric
space. This $S$-metric is generated by the usual metric on $%
\mathbb{R}
$. Let us consider the circle $C_{0,2}^{S}$ and define the self-mapping $%
T_{1}:%
\mathbb{R}
\rightarrow
\mathbb{R}
$ as%
\begin{equation*}
T_{1}x=\left\{
\begin{array}{ccc}
x & ; & x\in \{-1,1\} \\
10 & ; & \text{otherwise}%
\end{array}%
\right. \text{,}
\end{equation*}%
for all $x\in
\mathbb{R}
$. Then the self-mapping $T_{1}$ satisfies the conditions $($\ref{thm1_S1}$)$
and $($\ref{thm1_S2}$)$. Hence $C_{0,2}^{S}=\{-1,1\}$ is a fixed circle of $%
T_{1}$.

Notice that $C_{\frac{9}{2},11}^{S}=\{-1,10\}$ is another fixed circle of $%
T_{1}$ and so the fixed circle is not unique for a giving self-mapping.

On the other hand, if we consider the usual metric $d$ on $%
\mathbb{R}
$ then we obtain $C_{0,2}=\{-2,2\}$. The circle $C_{0,2}$ is not a fixed
circle of $T_{1}$.
\end{example}

\begin{example}
\label{exm13} Let $X=%
\mathbb{R}
^{2}$ and the function $\mathcal{S}:X^{3}\rightarrow \lbrack 0,\infty )$ be
defined by%
\begin{equation*}
\mathcal{S}(x,y,z)=\sum\limits_{i=1}^{2}\left( \left\vert
x_{i}-z_{i}\right\vert +\left\vert x_{i}+z_{i}-2y_{i}\right\vert \right)
\text{,}
\end{equation*}%
for all $x=(x_{1},x_{2})$, $y=(y_{1},y_{2})$ and $z=(z_{1},z_{2})$. Then it
can be easily seen that $\mathcal{S}$ is an $S$-metric on $%
\mathbb{R}
^{2}$, which is not generated by any metric, and the pair $\left(
\mathbb{R}
^{2},\mathcal{S}\right) $ is an $S$-metric space.

Let us consider the unit circle $C_{0,1}^{S}$ and define the self-mapping $%
T_{2}:%
\mathbb{R}
\rightarrow
\mathbb{R}
$ as%
\begin{equation*}
T_{2}x=\left\{
\begin{array}{ccc}
x & ; & x\in C_{0,1}^{S} \\
(1,0) & ; & \text{otherwise}%
\end{array}%
\right. \text{,}
\end{equation*}%
for all $x\in
\mathbb{R}
^{2}$. Then the self-mapping $T_{2}$ satisfies the conditions $($\ref%
{thm1_S1}$)$ and $($\ref{thm1_S2}$)$. Therefore $C_{0,1}^{S}$ is a fixed
circle of $T_{2}$ as shown in Figure \ref{fig:5}.
\end{example}

\begin{figure}[t]
\centering
\includegraphics[width=.9\linewidth]{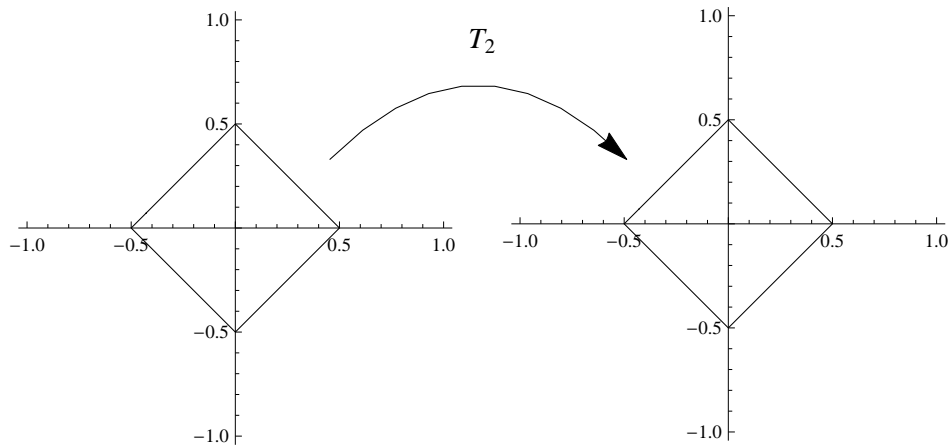}
\caption{\Small The fixed circle of $T_{2}$.}
\label{fig:5}
\end{figure}

In the following example, we give an example of a self-mapping which
satisfies the condition $($\ref{thm1_S1}$)$ and does not satisfy the
condition $($\ref{thm1_S2}$)$.

\begin{example}
\label{exm7} Let $X=%
\mathbb{R}
$ and the function $\mathcal{S}:X^{3}\rightarrow \lbrack 0,\infty )$ be
defined by%
\begin{equation*}
\mathcal{S}(x,y,z)=\left\vert x-z\right\vert +\left\vert x+z-2y\right\vert
\text{,}
\end{equation*}%
for all $x,y,z\in
\mathbb{R}
$ \cite{nihal2}. Then $\mathcal{S}$ is an $S$-metric which is not generated
by any metric and $(X,\mathcal{S})$ is an $S$-metric space. Let us consider
the circle $C_{0,3}^{S}$ and define the self-mapping $T_{3}:%
\mathbb{R}
\rightarrow
\mathbb{R}
$ as%
\begin{equation*}
T_{3}x=\left\{
\begin{array}{ccc}
-\dfrac{7}{2} & ; & x=-\dfrac{3}{2} \\
\dfrac{7}{2} & ; & x=\dfrac{3}{2} \\
7 & ; & \text{otherwise}%
\end{array}%
\right. \text{,}
\end{equation*}%
for all $x\in
\mathbb{R}
$. Then the self-mapping $T_{3}$ satisfies the condition $($\ref{thm1_S1}$)$
but does not satisfy the condition $($\ref{thm1_S2}$)$. Clearly $T_{3}$ does
not fix the circle $C_{0,3}^{S}$.
\end{example}

In the following example, we give an example of a self-mapping which
satisfies the condition $($\ref{thm1_S2}$)$ and does not satisfy the
condition $($\ref{thm1_S1}$)$.

\begin{example}
\label{exm8} Let $(X,\mathcal{S})$ be an $S$-metric space, $C_{x_{0},r}^{S}$
be a circle on $X$ and the self-mapping $T_{4}:X\rightarrow X$ be defined as%
\begin{equation*}
T_{4}x=x_{0}\text{,}
\end{equation*}%
for all $x\in X$. Then the self-mapping $T_{4}$ satisfies the condition $($%
\ref{thm1_S2}$)$ but does not satisfy the condition $($\ref{thm1_S1}$)$.
Clearly $T_{4}$ does not fix a circle $C_{x_{0},r}^{S}$.
\end{example}

Now we give another existence theorem for fixed circles.

\begin{theorem}
\label{thm2} Let $(X,\mathcal{S})$ be an $S$-metric space and $%
C_{x_{0},r}^{S}$ be any circle on $X$. Let the mapping $\varphi $ be defined
as $($\ref{phi mapping}$)$. If there exists a self-mapping $T:X\rightarrow X$
satisfying%
\begin{equation}
\mathcal{S}(x,x,Tx)\leq \varphi (x)-\varphi (Tx)  \label{thm2_S1}
\end{equation}%
and%
\begin{equation}
h\mathcal{S}(x,x,Tx)+\mathcal{S}(Tx,Tx,x_{0})\geq r\text{,}  \label{thm2_S2}
\end{equation}%
for all $x\in C_{x_{0},r}^{S}$ and some $h\in \lbrack 0,1),$ then $%
C_{x_{0},r}^{S}$ is a fixed circle of $T$.
\end{theorem}

\begin{proof}
Let $x\in C_{x_{0},r}^{S}$. On the contrary, assume that $x\neq Tx$. Then
using the conditions (\ref{thm2_S1}) and (\ref{thm2_S2}), we obtain%
\begin{eqnarray*}
\mathcal{S}(x,x,Tx) &\leq &\varphi (x)-\varphi (Tx) \\
&=&\mathcal{S}(x,x,x_{0})-\mathcal{S}(Tx,Tx,x_{0}) \\
&=&r-\mathcal{S}(Tx,Tx,x_{0}) \\
&\leq &h\mathcal{S}(x,x,Tx)+\mathcal{S}(Tx,Tx,x_{0})-\mathcal{S}(Tx,Tx,x_{0})
\\
&=&h\mathcal{S}(x,x,Tx)\text{,}
\end{eqnarray*}%
which is a contradiction since $h\in \lbrack 0,1)$. Hence we get $Tx=x$ and $%
C_{x_{0},r}^{S}$ is a fixed circle of $T$.
\end{proof}

\begin{remark}
\label{rem3} $1)$ Notice that the condition $($\ref{thm2_S1}$)$ guarantees
that $Tx$ is not in the exterior of the circle $C_{x_{0},r}^{S}$ for $x\in
C_{x_{0},r}^{S}$. Similarly, the condition $($\ref{thm2_S2}$)$ guarantees
that $Tx$ should be lie on or exterior or interior of the circle $%
C_{x_{0},r}^{S}$ for $x\in C_{x_{0},r}^{S}$. Hence $Tx$ should be lie on or
interior of the circle $C_{x_{0},r}^{S}$.

$2)$ If an $S$-metric is generated by any metric $d$, then Theorem \ref{thm2}
can be used on the corresponding metric space.
\end{remark}

Now we give some examples of self-mappings which have a fixed-circle.

\begin{example}
\label{exm9} Let $X=%
\mathbb{R}
$ and $(X,\mathcal{S})$ be the usual $S$-metric space. Let us consider the
circle $C_{1,2}^{S}=\{0,2\}$ and define the self-mapping $T_{5}:%
\mathbb{R}
\rightarrow
\mathbb{R}
$ as%
\begin{equation*}
T_{5}x=\left\{
\begin{array}{ccc}
e^{x}-1 & ; & x=0 \\
2x-2 & ; & x=2 \\
3 & ; & \text{otherwise}%
\end{array}%
\right. \text{,}
\end{equation*}%
for all $x\in
\mathbb{R}
$. Then the self-mapping $T_{5}$ satisfies the conditions $($\ref{thm2_S1}$)$
and $($\ref{thm2_S2}$)$. Hence $C_{1,2}^{S}$ is a fixed circle of $T_{5}$.

On the other hand, if we consider the usual metric $d$ on $%
\mathbb{R}
$ then we have $C_{1,2}=\{-1,3\}$. The circle $C_{1,2}$ is not a fixed
circle of $T_{5}$. But $C_{1,1}=\{0,2\}$ is a fixed circle of $T_{5}$ on $%
(X,d)$.
\end{example}

\begin{figure}[t]
\centering
\includegraphics[width=.9\linewidth]{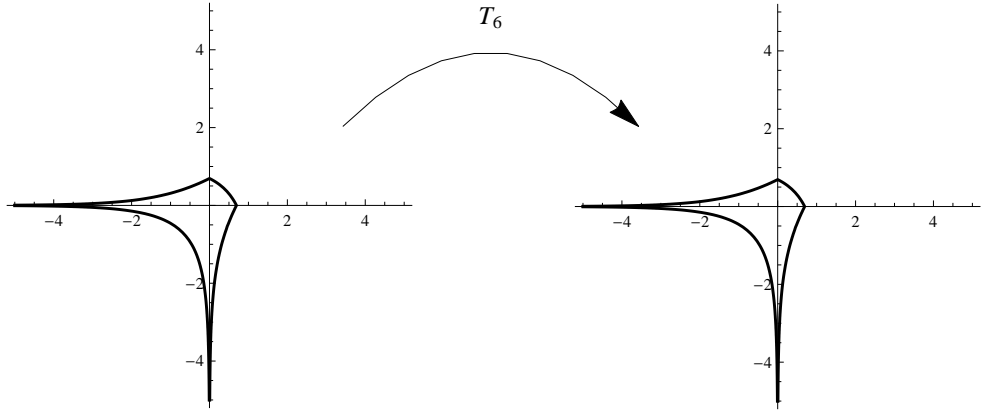}
\caption{\Small The fixed circle of $T_{6}$.}
\label{fig:6}
\end{figure}

\begin{example}
\label{exm14} Let $X=%
\mathbb{R}
^{2}$ and the function $\mathcal{S}:X^{3}\rightarrow \lbrack 0,\infty )$ be
defined by%
\begin{equation*}
\mathcal{S}(x,y,z)=\sum\limits_{i=1}^{2}\left( \left\vert
e^{x_{i}}-e^{z_{i}}\right\vert +\left\vert
e^{x_{i}}+e^{z_{i}}-2e^{y_{i}}\right\vert \right) \text{,}
\end{equation*}%
for all $x=(x_{1},x_{2})$, $y=(y_{1},y_{2})$ and $z=(z_{1},z_{2})$. Then it
can be easily checked that $\mathcal{S}$ is an $S$-metric on $%
\mathbb{R}
^{2}$, which is not generated by any metric, and the pair $\left(
\mathbb{R}
^{2},\mathcal{S}\right) $ is an $S$-metric space.

Let us consider the circle $C_{x_{0},r}^{S}$ centered at $x_{0}=(0,0)$ with
radius $r=2$ and define the self-mapping $T_{6}:%
\mathbb{R}
\rightarrow
\mathbb{R}
$ as%
\begin{equation*}
T_{6}x=\left\{
\begin{array}{ccc}
x & ; & x\in C_{0,2}^{S} \\
(\ln 2,0) & ; & \text{otherwise}%
\end{array}%
\right. \text{,}
\end{equation*}%
for all $x\in
\mathbb{R}
^{2}$. Then the self-mapping $T_{6}$ satisfies the conditions $($\ref%
{thm2_S1}$)$ and $($\ref{thm2_S2}$)$. Therefore $C_{0,2}^{S}$ is the fixed
circle of $T_{6}$ as shown in Figure \ref{fig:6}.
\end{example}

In the following example, we give an example of a self-mapping which
satisfies the condition $($\ref{thm2_S1}$)$ and does not satisfy the
condition $($\ref{thm2_S2}$)$.

\begin{example}
\label{exm10} Let $(X,\mathcal{S})$ be an $S$-metric space, $C_{x_{0},r}^{S}$
be a circle on $X$ and the self-mapping $T_{7}:X\rightarrow X$ be defined as%
\begin{equation*}
T_{7}x=x_{0}\text{,}
\end{equation*}%
for all $x\in X$. Then the self-mapping $T_{7}$ satisfies the condition $($%
\ref{thm2_S1}$)$ but does not satisfy the condition $($\ref{thm2_S2}$)$. It
can be easily seen that $T_{7}$ does not fix a circle $C_{x_{0},r}^{S}$.
\end{example}

In the following example, we give an example of a self-mapping which
satisfies the condition $($\ref{thm2_S2}$)$ and does not satisfy the
condition $($\ref{thm2_S1}$)$.

\begin{example}
\label{exm11} Let $X=%
\mathbb{R}
$ and $(X,\mathcal{S})$ be an $S$-metric space with $S$-metric defined as in
Example \ref{exm7}. Let us consider the unit circle $C_{0,1}^{S}$ and define
the self-mapping $T_{8}:%
\mathbb{R}
\rightarrow
\mathbb{R}
$ as%
\begin{equation*}
T_{8}x=1\text{,}
\end{equation*}%
for all $x\in
\mathbb{R}
$. Then the self-mapping $T_{8}$ satisfies the condition $($\ref{thm2_S2}$)$
but does not satisfy the condition $($\ref{thm2_S1}$)$. It can be easily
shown that $T_{8}$ does not fix the unit circle $C_{0,1}^{S}$.
\end{example}

Let $I_{X}:X\rightarrow X$ be the identity map defined as $I_{X}(x)=x$ for
all $x\in X$. Notice that the identity map satisfies the conditions $($\ref%
{thm1_S1}$)$ and $($\ref{thm1_S2}$)$ (resp. $($\ref{thm2_S1}$)$ and $($\ref%
{thm2_S2}$)$) in Theorem \ref{thm1} (resp. Theorem \ref{thm2}) for any
circle. Now we determine a condition which excludes the $I_{X}$ in Theorem %
\ref{thm1} and Theorem \ref{thm2}. We give the following theorem.

\begin{theorem}
\label{thm3} Let $(X,\mathcal{S})$ be an $S$-metric space, $T:X\rightarrow X$
be a self mapping having a fixed circle $C_{x_{0},r}^{S}$ and the mapping $%
\varphi $ be defined as $($\ref{phi mapping}$)$. The self-mapping $T$
satisfies the condition%
\begin{equation*}
(I_{S})\text{ \ \ \ \ \ \ \ \ }\mathcal{S}(x,x,Tx)\leq \frac{\varphi
(x)-\varphi (Tx)}{h}\text{,}
\end{equation*}%
for all $x\in X$ and some $h>2$ if and only if $T=I_{X}$.
\end{theorem}

\begin{proof}
Let $x\in X$ be an arbitrary element. Then using the inequality $(I_{S})$,
Lemma \ref{lem1} and triangle inequality, we obtain%
\begin{eqnarray*}
h\mathcal{S}(x,x,Tx) &\leq &\varphi (x)-\varphi (Tx) \\
&=&\mathcal{S}(x,x,x_{0})-\mathcal{S}(Tx,Tx,x_{0}) \\
&\leq &2\mathcal{S}(x,x,Tx)+\mathcal{S}(Tx,Tx,x_{0})-\mathcal{S}(Tx,Tx,x_{0})
\\
&=&2\mathcal{S}(x,x,Tx)
\end{eqnarray*}%
and so%
\begin{equation*}
(h-2)\mathcal{S}(x,x,Tx)\leq 0\text{.}
\end{equation*}%
Since $h>2$ it should be $\mathcal{S}(x,x,Tx)=0$ and so $Tx=x$.
Consequently, we obtain $T=I_{X}$.

Conversely, it is clear that the identity map $I_{X}$ satisfies the
condition $(I_{S})$.
\end{proof}

\begin{remark}
\label{rem4} $1)$ If a self-mapping $T$, which has a fixed circle, satisfies
the conditions $($\ref{thm1_S1}$)$ and $($\ref{thm1_S2}$)$ $($resp. $($\ref%
{thm2_S1}$)$ and $($\ref{thm2_S2}$))$ in Theorem \ref{thm1} $($resp. Theorem %
\ref{thm2}$)$ but does not satisfy the condition $(I_{S})$ in Theorem \ref%
{thm3} then the self-mapping $T$ can not be identity map.

$2)$ If an $S$-metric is generated by any metric $d$, then Theorem \ref{thm3}
can be used on the corresponding metric space.
\end{remark}

\subsection{\textbf{The uniqueness of fixed circles}}

In this section we investigate the uniqueness of fixed circles obtained in
the existence theorems. Let $(X,S)$ be an $S$-metric space. For any given
circles $C_{x_{0},r}^{S}$ and $C_{x_{1},\rho }^{S}$ on $X$, we notice that
there exists at least one self-mapping $T$ of $X$ such that $T$ fixes the
circles $C_{x_{0},r}^{S}$, $C_{x_{1},\rho }^{S}$. Indeed let us define the
mappings $\varphi _{1},\varphi _{2}:X\rightarrow \lbrack 0,\infty )$ as%
\begin{equation*}
\varphi _{1}(x)=\mathcal{S}(x,x,x_{0})
\end{equation*}%
and%
\begin{equation*}
\varphi _{2}(x)=\mathcal{S}(x,x,x_{1})\text{,}
\end{equation*}%
for all $x\in X$. If we define the self-mapping $T:X\rightarrow X$ as%
\begin{equation*}
Tx=\left\{
\begin{array}{ccc}
x & \text{;} & x\in C_{x_{0},r}^{S}\cup C_{x_{1},\rho }^{S} \\
\alpha & \text{;} & \text{otherwise}%
\end{array}%
\right. \text{,}
\end{equation*}%
for all $x\in X$, where $\alpha $ is a constant satisfying $S(\alpha ,\alpha
,x_{0})\neq r$ and $S(\alpha ,\alpha ,x_{1})\neq \rho $, it can be easily
that the self-mapping $T:X\rightarrow X$ satisfies the conditions $($\ref%
{thm1_S1}$)$ and $($\ref{thm1_S2}$)$ in Theorem \ref{thm1} $($resp. $($\ref%
{thm2_S1}$)$ and $($\ref{thm2_S2}$)$ in Theorem \ref{thm2}$)$ for the
circles $C_{x_{0},r}^{S}$ and $C_{x_{1},\rho }^{S}$ using the mappings $%
\varphi _{1}$ and $\varphi _{2}$, respectively. Hence $T$ fixes both of the
circles $C_{x_{0},r}^{S}$ and $C_{x_{1},\rho }^{S}$. By this way, the number
of fixed circles can be extended to any positive integer $n$ using the same
arguments.

In the following example, the self-mapping $T_{9}$ has two fixed circle.

\begin{example}
\label{exm12} Let $X=%
\mathbb{R}
$ and $(X,\mathcal{S})$ be an $S$-metric space with the $S$-metric defined
in Example \ref{exm7}. Let us consider the circles $C_{0,2}^{S}$, $%
C_{0,4}^{S}$ and define the self-mapping $T_{9}:%
\mathbb{R}
\rightarrow
\mathbb{R}
$ as%
\begin{equation*}
T_{9}x=\left\{
\begin{array}{ccc}
x & ; & x\in \{-2,-1,1,2\} \\
\alpha & ; & \text{otherwise}%
\end{array}%
\right. \text{,}
\end{equation*}%
for all $x\in X$ where $\alpha \in X$. Then the conditions $($\ref{thm1_S1}$%
) $ and $($\ref{thm1_S2}$)$ are satisfied by $T_{9}$ for the circles $%
C_{0,2}^{S}$ and $C_{0,4}^{S}$, respectively. Consequently, $C_{0,2}^{S}$
and $C_{0,4}^{S}$ are the fixed circles of $T_{9}$.
\end{example}

Now we investigate uniqueness conditions for the fixed circles in Theorem %
\ref{thm1} using Rhoades' contractive condition on $S$-metric spaces.

\begin{theorem}
\label{thm4} Let $(X,\mathcal{S})$ be an $S$-metric space and $%
C_{x_{0},r}^{S}$ be any circle on $X$. Let $T:X\rightarrow X$ be a
self-mapping satisfying the conditions $($\ref{thm1_S1}$)$ and $($\ref%
{thm1_S2}$)$ given in Theorem \ref{thm1}. If the contractive condition%
\begin{eqnarray}
\mathcal{S}(Tx,Tx,Ty) &<&\max \{\mathcal{S}(x,x,y),\mathcal{S}(Tx,Tx,x),%
\mathcal{S}(Ty,Ty,y),  \label{Rhoades} \\
&&\mathcal{S}(Ty,Ty,x),\mathcal{S}(Tx,Tx,y)\}\text{,}  \notag
\end{eqnarray}%
is satisfied for all $x\in C_{x_{0},r}^{S}$, $y\in X\setminus
C_{x_{0},r}^{S} $ by $T$, then $C_{x_{0},r}^{S}$ is the unique fixed circle
of $T$.
\end{theorem}

\begin{proof}
Suppose that there exist two fixed circles $C_{x_{0},r}^{S}$ and $%
C_{x_{1},\rho }^{S}$ of the self-mapping $T$, that is, $T$ satisfies the
conditions $($\ref{thm1_S1}$)$ and $($\ref{thm1_S2}$)$ for each circles $%
C_{x_{0},r}^{S}$ and $C_{x_{1},\rho }^{S}$. Let $x\in C_{x_{0},r}^{S}$ and $%
y\in C_{x_{1},\rho }^{S}$ be arbitrary points with $x\neq y$. Using the
contractive condition $($\ref{Rhoades}$)$ we obtain%
\begin{eqnarray*}
\mathcal{S}(x,x,y) &=&\mathcal{S}(Tx,Tx,Ty)<\max \{\mathcal{S}(x,x,y),%
\mathcal{S}(Tx,Tx,x),\mathcal{S}(Ty,Ty,y), \\
&&\mathcal{S}(Ty,Ty,x),\mathcal{S}(Tx,Tx,y)\} \\
&=&\mathcal{S}(x,x,y)\text{,}
\end{eqnarray*}%
which is a contradiction. Hence it should be $x=y$. Consequently, $%
C_{x_{0},r}^{S}$ is the unique fixed circle of $T$.
\end{proof}

In the following example we show that $C_{x_{0},r}^{S}$ is not necessarily
unique in Theorem \ref{thm2}.

\begin{example}
\label{exm15} Let $(X,\mathcal{S})$ be an $S$-metric space and $%
C_{x_{1},r_{1}}$,$\cdots $, $C_{x_{n},r_{n}}$ be any circles on $X$. Let us
define the self-mapping $T_{10}:X\rightarrow X$ as%
\begin{equation*}
T_{10}x=\left\{
\begin{array}{ccc}
x & \text{;} & x\in \bigcup\limits_{i=1}^{n}C_{x_{i},r_{i}} \\
x_{0} & \text{;} & \text{otherwise}%
\end{array}%
\right. \text{,}
\end{equation*}%
for all $x\in X$, where $x_{0}$ is a constant in $X$. Then it can be easily
checked that the conditions $($\ref{thm2_S1}$)$ and $($\ref{thm2_S2}$)$ are
satisfied by $T_{10}$ for the circles $C_{x_{1},r_{1}}$,$\cdots $, $%
C_{x_{n},r_{n}}$, respectively. Consequently, the circles $C_{x_{1},r_{1}}$,$%
\cdots $, $C_{x_{n},r_{n}}$ are fixed circles of $T_{10}$. Notice that these
circles do not have to be disjoint.
\end{example}

Now we give the following uniqueness theorem for the fixed circles in
Theorem \ref{thm2} using the notion of diameter on $S$-metric spaces.

\begin{theorem}
\label{thm5} Let $(X,\mathcal{S})$ be an $S$-metric space, $C_{x_{0},r}^{S}$
be any circle on $X$, $U_{x}=\{T^{n}x:n\in
\mathbb{N}
\}$, $U_{y}=\{T^{n}y:n\in
\mathbb{N}
\}$, $diam\{U_{x}\}<\infty $ and $diam\{U_{y}\}<\infty $. Let $%
T:X\rightarrow X$ be a self-mapping satisfying the conditions $($\ref%
{thm2_S1}$)$ and $($\ref{thm2_S2}$)$ given in Theorem \ref{thm2}. If the
contractive condition%
\begin{equation}
\mathcal{S}(Tx,Tx,Ty)<diam\{U_{x}\cup U_{y}\}\text{,}  \label{diameter}
\end{equation}%
is satisfied for all $x\in C_{x_{0},r}^{S}$, $y\in X\setminus
C_{x_{0},r}^{S} $ by $T$, then $C_{x_{0},r}^{S}$ is the unique fixed circle
of $T$.
\end{theorem}

\begin{proof}
Assume that there exist two fixed circles $C_{x_{0},r}^{S}$ and $%
C_{x_{1},\rho }^{S}$ of the self-mapping $T$, that is, $T$ satisfies the
conditions $($\ref{thm2_S1}$)$ and $($\ref{thm2_S2}$)$ for each circles $%
C_{x_{0},r}^{S}$ and $C_{x_{1},\rho }^{S}$. Let $x\in C_{x_{0},r}^{S}$ and $%
y\in C_{x_{1},\rho }^{S}$ be arbitrary points with $x\neq y$. Using the
contractive condition $($\ref{diameter}$)$ we obtain%
\begin{equation*}
\mathcal{S}(x,x,y)=\mathcal{S}(Tx,Tx,Ty)<diam\{U_{x}\cup U_{y}\}=\mathcal{S}%
(x,x,y),
\end{equation*}%
which is a contradiction. Hence it should be $x=y$. Consequently, $%
C_{x_{0},r}^{S}$ is the unique fixed circle of $T$.
\end{proof}

\subsection{Infinity of fixed circles}

In this section we give a new approach to obtain fixed-circle results. To do
this, let us denote by $R_{S}(x,y)$ the right side of the inequality $(S25)$%
. Using the number $R_{S}(x,y)$, we obtain the following theorem. This
theorem generates many (finite or infinite) fixed circle for a given
self-mapping.

\begin{theorem}
\label{thm6} Let $(X,\mathcal{S})$ be an $S$-metric space, $T:X\rightarrow X$
be a self-mapping and $r=\min \left\{ \mathcal{S}(Tx,Tx,x):Tx\neq x\right\} $%
. If there exists a point $x_{0}\in X$ satisfying%
\begin{equation}
\mathcal{S}(x,x,Tx)<R_{S}(x,x_{0})\text{ for all }x\in X\text{ when }%
\mathcal{S}(Tx,Tx,x)>0  \label{eqn1}
\end{equation}%
and%
\begin{equation}
\mathcal{S}(Tx,Tx,x_{0})=r\text{ for all }x\in C_{x_{0},r}^{S}\text{,}
\label{eqn2}
\end{equation}%
then $C_{x_{0},r}^{S}$ is a fixed circle of $T$. The self-mapping $T$ also
fixes the closed ball $B_{S}[x_{0},r]$.
\end{theorem}

\begin{proof}
Let $x\in C_{x_{0},r}^{S}$ and $Tx\neq x$. Then using the inequality (\ref%
{eqn1}) and Lemma \ref{lem1}, we get%
\begin{eqnarray}
\mathcal{S}(x,x,Tx) &<&R_{S}(x,x_{0})  \notag \\
&=&\max \left\{
\begin{array}{c}
\mathcal{S}(x,x,x_{0}),\mathcal{S}(Tx,Tx,x),\mathcal{S}(Tx_{0},Tx_{0},x_{0}),
\\
\mathcal{S}(Tx_{0},Tx_{0},x),\mathcal{S}(Tx,Tx,x_{0})%
\end{array}%
\right\} \text{.}  \label{eqn3}
\end{eqnarray}%
At first, using the inequality (\ref{eqn3}) and Lemma \ref{lem1}, we show $%
Tx_{0}=x_{0}$. Suppose that $Tx_{0}\neq x_{0}$. For $x=x_{0}$, we obtain%
\begin{eqnarray*}
\mathcal{S}(x_{0},x_{0},Tx_{0}) &<&R_{S}(x_{0},x_{0}) \\
&=&\max \left\{
\begin{array}{c}
\mathcal{S}(x_{0},x_{0},x_{0}),\mathcal{S}(Tx_{0},Tx_{0},x_{0}),\mathcal{S}%
(Tx_{0},Tx_{0},x_{0}), \\
\mathcal{S}(Tx_{0},Tx_{0},x_{0}),\mathcal{S}(Tx_{0},Tx_{0},x_{0})%
\end{array}%
\right\} \\
&=&\mathcal{S}(Tx_{0},Tx_{0},x_{0})=\mathcal{S}(x_{0},x_{0},Tx_{0})\text{,}
\end{eqnarray*}%
a contradiction. It should be $Tx_{0}=x_{0}$. Then by the inequality (\ref%
{eqn3}), the condition (\ref{eqn2}) and Lemma \ref{lem1}, we have%
\begin{eqnarray*}
\mathcal{S}(x,x,Tx) &<&\max \left\{
\begin{array}{c}
\mathcal{S}(x,x,x_{0}),\mathcal{S}(Tx,Tx,x),\mathcal{S}(x_{0},x_{0},x_{0}),
\\
\mathcal{S}(x_{0},x_{0},x),\mathcal{S}(Tx,Tx,x_{0})%
\end{array}%
\right\} \\
&=&\max \left\{ r,\mathcal{S}(Tx,Tx,x)\right\} =\mathcal{S}(Tx,Tx,x)=%
\mathcal{S}(x,x,Tx)\text{,}
\end{eqnarray*}%
a contradiction. Therefore we get $Tx=x$, that is, $C_{x_{0},r}^{S}$ is a
fixed circle of $T$.

Finally we prove that $T$ fixes the closed ball $B_{S}[x_{0},r]$. To do
this, we show that $T$ fixes any circle $C_{x_{0},\rho }^{S}$ with $\rho <r$%
. Let $x\in C_{x_{0},\rho }^{S}$ and $Tx\neq x$. From the similar arguments
used in the above, we have $Tx=x$.
\end{proof}

We give the following example.

\begin{example}
\label{exm1} Let $X=%
\mathbb{R}
$ be the usual $S$-metric space. Let us define the self-mapping $T:%
\mathbb{R}
\rightarrow
\mathbb{R}
$ as%
\begin{equation*}
Tx=\left\{
\begin{array}{ccc}
x & ; & \left\vert x\right\vert <3 \\
x+2 & ; & \left\vert x\right\vert \geq 3%
\end{array}%
\right. \text{,}
\end{equation*}%
for all $x\in
\mathbb{R}
$. The self-mapping $T$ satisfies the conditions of Theorem \ref{thm6} with $%
x_{0}=0$. Indeed, we get%
\begin{equation*}
\mathcal{S}(x,x,Tx)=2\left\vert x-Tx\right\vert =4>0\text{,}
\end{equation*}%
for all $x\in
\mathbb{R}
$ such that $\left\vert x\right\vert \geq 3$. Then we have%
\begin{eqnarray*}
R_{S}(x,0) &=&\max \left\{ \mathcal{S}(x,x,0),\mathcal{S}(Tx,Tx,x),\mathcal{%
S(}0,0,0\mathcal{)},\mathcal{S(}0,0,x\mathcal{)},\mathcal{S(}Tx,Tx,0\mathcal{%
)}\right\} \\
&=&\max \left\{ 2\left\vert x\right\vert ,4,0,2\left\vert x\right\vert
,2\left\vert x+2\right\vert \right\} \\
&=&\max \left\{ 2\left\vert x\right\vert ,2\left\vert x+2\right\vert \right\}
\end{eqnarray*}%
and so%
\begin{equation*}
\mathcal{S}(x,x,Tx)<R_{S}(x,0)\text{.}
\end{equation*}%
Therefore the condition $($\ref{eqn1}$)$ is satisfied. It can be easily seen
that the condition $($\ref{eqn2}$)$ is satisfied by the definition of $T$.
We also obtain%
\begin{equation*}
r=\min \left\{ \mathcal{S}(Tx,Tx,x):Tx\neq x\right\} =4\text{.}
\end{equation*}%
Consequently, $T$ fixes the circle $C_{0,4}^{S}=\left\{ x\in
\mathbb{R}
:\left\vert x\right\vert =2\right\} $ and the closed ball $%
B_{S}[0,4]=\left\{ x\in
\mathbb{R}
:\left\vert x\right\vert \leq 2\right\} $.
\end{example}

\begin{remark}
\label{rem5} $1)$ Notice that the condition $($\ref{eqn2}$)$ guarantees that
$Tx\in C_{x_{0},r}^{S}$ for each $x\in C_{x_{0},r}^{S}$ and so $%
T(C_{x_{0},r}^{S})\subset C_{x_{0},r}^{S}$.

$2)$ The self-mapping $T$ defined in Example \ref{exm1} has other fixed
circles. Theorem \ref{thm6} gives us some of these circles.

$3)$ A self-mapping $T$ can fix infinitely many circles $($see Example \ref%
{exm1}$)$.
\end{remark}

The converse statement is not always true as seen following example.

\begin{example}
\label{exm2} Let $x_{0}\in X$ be any point. If we define the self-mapping $%
T:X\rightarrow X$ as%
\begin{equation*}
Tx=\left\{
\begin{array}{ccc}
x & ; & x\in B_{S}[x_{0},\mu ] \\
x_{0} & ; & x\notin B_{S}[x_{0},\mu ]%
\end{array}%
\right. \text{,}
\end{equation*}%
for all $x\in X$ with $\mu >0$, then $T$ does not satisfies the condition $($%
\ref{eqn1}$)$, but $T$ fixes every circle $C_{x_{0},\rho }^{S}$ with $\rho
\leq \mu $.
\end{example}


\begin{thebibliography}{99}
\bibitem{Banach} S. Banach, Sur les operations dans les ensembles abstraits
et leur application aux equations integrals, \textit{Fund. Math.} 2 (1922),
133-181.

\bibitem{Ciesielski-2007} K. Ciesielski, On Stefan Banach and some of his
results, \textit{Banach J. Math. Anal.} 1 (2007), no.1, 1-10.

\bibitem{Hieu} N. T. Hieu, N. T. Ly and N. V. Dung, A Generalization of
Ciric Quasi-Contractions for Maps on $S$-Metric Spaces, \textit{Thai Journal
of Mathematics} 13 (2015), no.2, 369-380.

\bibitem{Ozdemir-2011} N. \"{O}zdemir, B. B. \.{I}skender and N. Y. \"{O}zg%
\"{u}r, Complex valued neural network with M\"{o}bius activation function,
\textit{Commun. Nonlinear Sci. Numer. Simul.} 16 (2011), no.12, 4698-4703.

\bibitem{nihal} N. Y. \"{O}zg\"{u}r and N. Ta\c{s}, Some fixed point
theorems on $S$-metric spaces{,} \textit{Mat. Vesnik}{\ 69} (2017), no.1,
39-52.

\bibitem{nihal2} N. Y. \"{O}zg\"{u}r and N. Ta\c{s}, Some new contractive
mappings on $S$-metric spaces and their relationships with the mapping $%
(S25) $, \textit{Math. Sci.} 11 (2017), no.1, 7-16.

\bibitem{nihal3} N. Y. \"{O}zg\"{u}r and N. Ta\c{s}, Some generalizations of
fixed point theorems on $S$-metric spaces, \textit{Essays in Mathematics and
Its Applications in Honor of Vladimir Arnold}, New York, Springer, 2016.

\bibitem{nihal4} N. Y. \"{O}zg\"{u}r and N. Ta\c{s}, Some fixed-circle
theorems on metric spaces, Bull. Malays. Math. Sci. Soc., (2017).
https://doi.org/10.1007/s40840-017-0555-z

\bibitem{Ozgur-Aip} N. Y. \"{O}zg\"{u}r and N. Ta\c{s}, Some fixed-circle
theorems and discontinuity at fixed circle, AIP Conference Proceedings 1926,
020048 (2018).

\bibitem{Rhoades} B. E. Rhoades, A comparison of various definitions of
contractive mappings, \textit{Trans. Amer. Math. Soc.} 226 (1977), 257-290.

\bibitem{Sedghi-2012} S. Sedghi, N. Shobe and A. Aliouche, A Generalization
of Fixed Point Theorems in $S$-Metric Spaces, \textit{Mat. Vesnik} 64
(2012), no.3, 258-266.

\bibitem{Sedghi-2014} S. Sedghi and N. V. Dung, Fixed Point Theorems on $S$%
-Metric Spaces{,} \textit{Mat. Vesnik}{\ 66} (2014), no.1, 113-124.

\bibitem{tez} N. Ta\c{s}, Fixed point theorems and their various
applications, Ph. D. Thesis, 2017.
\end{thebibliography}
\end{document}